\documentclass{amsart}
\usepackage{amssymb}
\usepackage{amsfonts}

\newtheorem{theorem}{Theorem}
\theoremstyle{plain}

\newtheorem{corollary}{Corollary}

\newtheorem{definition}{Definition}
\newtheorem{example}{Example}

\newtheorem{lemma}{Lemma}

\numberwithin{equation}{section}
\input{tcilatex}

\begin{document}
\title[The relation between parameter curves and...]{The relation between
parameter curves and lines of curvature on canal surfaces}
\author{Fatih Do\u{g}an and Yusuf Yayl\i }
\address{\textit{Current Adress: Fatih} \textit{DO\u{G}AN}, \textit{Yusuf} 
\textit{YAYLI}, \textit{Ankara University, Department of Mathematics, 06100
Tando\u{g}an, Ankara, Turkey}}
\email{mathfdogan@hotmail.com, yayli@science.ankara.edu.tr}
\date{}

\begin{abstract}
A canal surface is the envelope of a moving sphere with varying radius,
defined by the trajectory $C(t)$ (spine curve) of its center and a radius
function $r(t)$. In this paper, we investigate when parameter curves of the
canal surface are also lines of curvature. Last of all, for special spine
curves we obtain the radius function of canal surfaces.
\end{abstract}

\maketitle

\section{Introduction}

A canal surface is defined as the envelope of a\ family of one parameter
spheres and is useful for representing long thin objects e.g. pipes, poles,
ropes, 3D fonts or intestines of body. Canal surfaces are also frequently
used in solid and surface modelling for CAD/CAM. Representative examples are
natural quadrics, tori, tubular surfaces and Dupin cyclides.

A curve on a surface which has the property that its tangent at each of its
points $p$ coincides with a principal direction at $p$ is called line of
curvature put differently%
\begin{equation*}
S(T)=k_{n}T,
\end{equation*}%
where $S$ is the shape operator of the surface, $T$ is the tangent vector
field of the curve and $k_{n}$ is the normal curvature along the curve on
the surface.

Maekawa $et.al.$ $[8]$ researched necessary and sufficient conditions for
the regularity of pipe (tube) surfaces. More recently, Xu $et.al.$ $[10]$
studied these conditions for canal surfaces and examined principle geometric
properties of canal surfaces like computing the area and Gaussian curvature
of them.

Gross $[2]$ gave the concept of generalized tubes (briefly GT) and
classified them in two types as ZGT and CGT. Here, ZGT refers to the spine
curve (the axis) that has torsion-free and CGT refers to tube that has
circular cross sections. He investigated the properties of GT and showed
that parameter curves of a generalized tube are also lines of curvature if
and only if the spine curve is planar. In this study, we examine when $s-$
and $\theta -$parameter curves of the canal surface are also lines of
curvature.

This paper is organized as follows. In section 2, we introduce a canal
surface and give basic notions about it. Then we observe parameter curves of
generalized tubes which are also lines of curvature and look through this
property on canal surfaces in section 3. Furthermore, we obtain the radius
function of the canal surface that $s-$parameter curves are lines of
curvature at the same time. In section 4, we conclude this paper.

\section{Preliminaries}

First of all, we present a canal surface and give some coefficients of the
first and second fundamental form of it. Subsequently, we mention tube and
generalized tube.

\begin{definition}
Canal surface is defined as the envelope of a\ family of one parameter
spheres. Alternatively, a canal surface is the envelope of a moving sphere
with varying radius, defined by the trajectory $C(t)$ (spine curve) of its
center and a radius function $r(t)$. When $r^{^{\prime }}(t)<\left \Vert
C^{^{\prime }}(t)\right \Vert $, the canal surface is regular and
parametrized as follows.%
\begin{equation*}
(2.1)\text{ \ }K(t,\theta )=C(t)-r(t)r^{^{\prime }}(t)\frac{C^{^{\prime }}(t)%
}{\left \Vert C^{^{\prime }}(t)\right \Vert ^{2}}\mp r(t)\frac{\sqrt{\left
\Vert C^{^{\prime }}(t)\right \Vert ^{2}-r^{^{\prime }}(t)^{2}}}{\left \Vert
C^{^{\prime }}(t)\right \Vert }\left( \cos \theta N+\sin \theta B\right) .
\end{equation*}%
If the spine curve $C(t)$ has arclenght parametrization ($\left \Vert
C^{^{\prime }}(t)\right \Vert =1$), then the canal surface is reparametrized
as $[1]$%
\begin{equation*}
(2.2)\text{ \  \  \  \  \  \  \  \ }K(s,\theta )=C(s)-r(s)r^{^{\prime }}(s)T(s)\mp
r(s)\sqrt{1-r^{^{\prime }}(s)^{2}}\cos \theta N(s)+\sin \theta B(s)
\end{equation*}%
where $\{T,$ $N,$ $B\}$ is the Frenet frame of the spine curve. The normal
vector field and some coefficients of first and second fundamental form of
the canal surface are as follows $[10]$.%
\begin{eqnarray*}
N &=&\frac{K_{s}\times K_{\theta }}{\left \Vert K_{s}\times K_{\theta
}\right \Vert } \\
N &=&\frac{\left( g^{^{\prime }}-h\kappa \cos \theta \right) T+\left( \kappa
g\cos \theta +h^{^{\prime }}-1\right) \cos \theta N+\left( \kappa g\cos
\theta +h^{^{\prime }}-1\right) \sin \theta B}{\sqrt{\left( \kappa g\cos
\theta +h^{^{\prime }}-1\right) ^{2}+\left( g^{^{\prime }}-h\kappa \cos
\theta \right) ^{2}}}
\end{eqnarray*}%
\begin{equation*}
E=K_{s}\centerdot K_{s}=\left( \kappa g\cos \theta +h^{^{\prime }}-1\right)
^{2}+\left( g\tau +h\kappa \sin \theta \right) ^{2}+\left( g^{^{\prime
}}-h\kappa \cos \theta \right) ^{2}
\end{equation*}%
\begin{equation}
F=K_{s}\centerdot K_{\theta }=g\left( g\tau +h\kappa \sin \theta \right) 
\tag{2.3}
\end{equation}%
\begin{equation*}
G=K_{\theta }\centerdot K_{\theta }=g^{2}
\end{equation*}%
\begin{equation}
\left \Vert K_{s}\times K_{\theta }\right \Vert ^{2}=EG-F^{2}=g^{2}\left(
\left( \kappa g\cos \theta +h^{^{\prime }}-1\right) ^{2}+\left( g^{^{\prime
}}-h\kappa \cos \theta \right) ^{2}\right)  \tag{2.4}
\end{equation}%
\begin{eqnarray*}
f &=&N\centerdot K_{s\theta }=\frac{1}{\sqrt{\left( \kappa g\cos \theta
+h^{^{\prime }}-1\right) ^{2}+\left( g^{^{\prime }}-h\kappa \cos \theta
\right) ^{2}}} \\
&&\times \left[ \left( g^{^{\prime }}-h\kappa \cos \theta \right) \kappa
g\sin \theta -\tau g\left( \kappa g\cos \theta +h^{^{\prime }}-1\right) %
\right] ,
\end{eqnarray*}%
where $h=r(s)r^{^{\prime }}(s)\neq 0$ and $g=r(s)\sqrt{1-r^{^{\prime
^{2}}}(s)}\neq 0$. If the radius function $r(s)=r$ is a constant, then the
canal surface is called a tube or pipe surface and it is written as%
\begin{equation*}
K(s,\theta )=C(s)+r\left( \cos \theta N(s)+\sin \theta B(s)\right) .
\end{equation*}
\end{definition}

\section{Some Characterizations for Lines of Curvature on Canal Surface}

In this section, to begin with we inspect generalized tube. We observe when
parameter curves are also lines of curvature on a generalized tube $[2]$.
After that, we investigate the same case on a canal surface. Finally, we get
some characterizations of canal surfaces around special spine curves.

\begin{definition}
The parameterization of generalized tube around the spine curve $\Gamma (s)$
is%
\begin{equation}
X(s,\theta )=\Gamma (s)+u(\theta )\left( \cos \theta N(s)+\sin \theta
B(s)\right) ,\text{ }0\leq \theta <2\pi  \tag{3.1}
\end{equation}%
where $u$ is twice differentiable, $u(\theta )>0$ and $u(0)=u(2\pi )$.
\end{definition}

\begin{definition}
Let $M$ be a surface and let the curve $\alpha :I\subset R\longrightarrow M$%
. Then%
\begin{equation*}
(fE-eF)(u^{^{\prime }})^{2}+(gE-eG)u^{^{\prime }}v^{^{\prime
}}+(gF-fG)(v^{^{\prime }})^{2}=0
\end{equation*}%
is called as the differential equation of lines of curvature on $M$ $[6]$.
\end{definition}

\begin{theorem}
A necessary and sufficient condition for the parameter curves of a surface
to be lines of curvature in a neighborhood of a nonumbilical point is that $%
F=f=0$ where $F$ and $f$ are the respective the first and second fundamental
coefficients $[6]$.
\end{theorem}

\begin{proof}
Weingarten equations are given by%
\begin{eqnarray*}
-S(x_{u}) &=&U_{u}=\frac{fF-eG}{EG-F^{2}}x_{u}+\frac{eF-fE}{EG-F^{2}}x_{v} \\
-S(x_{v}) &=&U_{v}=\frac{gF-fG}{EG-F^{2}}x_{u}+\frac{fF-gE}{EG-F^{2}}x_{v}.
\end{eqnarray*}%
where $E,$ $F,$ $G$ and $e,$ $f,$ $g$ are coefficients of the first and
second fundamental form of a surface, respectively.
\end{proof}

($\Longrightarrow $) Assume that parameter curves in a neighborhood of a
nonumbilical point of a surface are also lines of curvature. In this case,
from the Definition (3) and Weingarten equations we get%
\begin{eqnarray*}
S(x_{u}) &=&-\frac{fF-eG}{EG-F^{2}}x_{u} \\
S(x_{v}) &=&-\frac{fF-gE}{EG-F^{2}}x_{v}
\end{eqnarray*}%
in other words%
\begin{eqnarray*}
eF-fE &=&0 \\
gF-fG &=&0.
\end{eqnarray*}%
From this, we have $F=f=0$.\newline
($\Longleftarrow $) Let $F=f=0$ in a neighborhood of a nonumbilical point of
a surface. By the Weingarten equations it follows that%
\begin{eqnarray*}
S(x_{u}) &=&\frac{e}{E}x_{u} \\
S(x_{v}) &=&\frac{g}{G}x_{v.}
\end{eqnarray*}%
Then, according to the definition of line of curvature $u-$and $v-$parameter
curves become lines of curvature concurrently.

Now, we give some coefficients of the first and second fundamental form of a
generalized tube. The normal vector field $N$ of a generalized tube can be
computed as the cross product of tangent vectors of $\theta -$parameter
curve and $s-$parameter curve and vice versa.%
\begin{eqnarray*}
X_{\theta } &=&\left( 1-\kappa u\cos \theta \right) T-u\tau \sin \theta
N+u\tau \cos \theta B \\
X_{s} &=&(u^{^{\prime }}\cos \theta -u\sin \theta )N+(u^{^{\prime }}\sin
\theta +u\cos \theta )B \\
N &=&X_{\theta }\times X_{s}=uu^{^{\prime }}\tau T+\left( 1-\kappa u\cos
\theta \right) \left[ 
\begin{array}{c}
\left( u^{^{\prime }}\sin \theta +u\cos \theta \right) N \\ 
+\left( u\sin \theta -u^{^{\prime }}\cos \theta \right) B%
\end{array}%
\right] ,
\end{eqnarray*}%
where $\kappa $ and $\tau $ are the curvature and the torsion of the spine
curve, respectively. Then%
\begin{gather}
F=X_{s}\centerdot X_{\theta }=u^{2}\tau  \tag{3.2} \\
f=N\centerdot X_{s\theta }=\frac{1}{\left \Vert N\right \Vert }\tau \left[
\kappa uu^{^{\prime }}\left( u\sin \theta -u^{^{\prime }}\cos \theta \right)
-\left( 1-\kappa u\cos \theta \right) (u^{2}+u^{^{\prime }2})\right] . 
\notag
\end{gather}

\begin{theorem}
The parameter curves of a generalized tube are lines of curvature if and
only if the axis $\Gamma $ is torsion-free, i.e., $\tau =0$.
\end{theorem}

\begin{proof}
Let the spine curve $\Gamma $ be a plane curve that is $\tau =0$. In that
case, by Eq (3.2) $F=f=0$. Conversely, let the parameter curves be also
lines of curvature. Therefore, from Theorem (1)\ $F=f=0$. Since $F=u^{2}\tau
=0$ and $u>0$, we have $\tau =0$, i.e, $\Gamma $ is a plane curve.\newline
From now on, we will look into the same case on a canal surface. For one
thing, we give an important lemma as regards regularity of a canal surface.
\end{proof}

\begin{lemma}[{$[10]$}]
For a canal surface, when $\kappa (s_{0})g(s_{0})\cos \theta
_{0}+h^{^{\prime }}(s_{0})-1=0$, $g^{^{\prime }}(s_{0})-h(s_{0})\kappa
(s_{0})\cos \theta _{0}=0$ where $s_{0}\in $ $\left[ 0,l\right] $ and $%
\theta _{0}\in \lbrack 0,2\pi )$.
\end{lemma}

\begin{proof}
Since $h=rr^{^{\prime }}\neq 0$ and $g=r\sqrt{1-r^{^{\prime }}{}^{2}}\neq 0$%
, we obtain%
\begin{equation*}
h(s_{0})(h^{^{\prime }}(s_{0})-1)=-g(s_{0})g^{^{\prime }}(s_{0}).
\end{equation*}%
If $\kappa (s_{0})g(s_{0})\cos \theta _{0}+h^{^{\prime }}(s_{0})-1=0$, then%
\begin{eqnarray*}
h(s_{0})\left( \kappa (s_{0})g(s_{0})\cos \theta _{0}+h^{^{\prime
}}(s_{0})-1\right) &=&0 \\
h(s_{0})\kappa (s_{0})g(s_{0})\cos \theta _{0}+h(s_{0})(h^{^{\prime
}}(s_{0})-1) &=&0 \\
h(s_{0})\kappa (s_{0})g(s_{0})\cos \theta _{0}-g(s_{0})g^{^{\prime }}(s_{0})
&=&0 \\
g(s_{0})\left( h(s_{0})\kappa (s_{0})\cos \theta _{0}-g^{^{\prime
}}(s_{0})\right) &=&0.
\end{eqnarray*}%
In the last equation, as $g(s_{0})\neq 0$, $h(s_{0})\kappa (s_{0})\cos
\theta _{0}-g^{^{\prime }}(s_{0})=0$. This completes the proof.\newline
Thus, from Eq (2.4) and Lemma (1) it follows that $K_{s}\times K_{\theta }=0$%
. Then, the canal surface is singular at the points $p=K(s_{0},\theta _{0})$
that is to say, when $\kappa g\cos \theta +h^{^{\prime }}-1\neq 0$ the canal
surface is regular.
\end{proof}

\begin{theorem}
For a regular canal surface,%
\begin{equation*}
F=0\iff f=0.
\end{equation*}
\end{theorem}

\begin{proof}
Assume that $F=0$. Then, by Eq (2.3) we have $g\left( g\tau +h\kappa \sin
\theta \right) =0$. In that $g\neq 0$, $g\tau =-h\kappa \sin \theta $. If we
substitute the last equality in the expression of $f$, we gather that%
\begin{equation*}
f=\frac{1}{\sqrt{\left( \kappa g\cos \theta +h^{^{\prime }}-1\right)
^{2}+\left( g^{^{\prime }}-h\kappa \cos \theta \right) ^{2}}}\left[ \left(
gg^{^{\prime }}+h(h^{^{\prime }}-1)\right) \kappa \sin \theta \right] .
\end{equation*}%
Since $gg^{^{\prime }}+h(h^{^{\prime }}-1)=0$, $f=0$. On the contrary,
assume that $f=0$. In this case,%
\begin{equation*}
\left( g^{^{\prime }}-h\kappa \cos \theta \right) \kappa g\sin \theta -\tau
g\left( \kappa g\cos \theta +h^{^{\prime }}-1\right) =0.
\end{equation*}%
If we arrange this equality, due to the fact that $gg^{^{\prime
}}+h(h^{^{\prime }}-1)=0$ it concludes%
\begin{equation*}
\left( \kappa g\cos \theta +h^{^{\prime }}-1\right) \left( g\tau +h\kappa
\sin \theta \right) =0.
\end{equation*}%
Since the canal surface is regular, $\kappa g\cos \theta +h^{^{\prime
}}-1\neq 0$. Then $g\tau +h\kappa \sin \theta =0$ and thus from Eq (2.3) $%
F=0 $.
\end{proof}

\begin{corollary}
The parameter curves are also lines of curvature on a regular canal surface
if and only if $g\tau +h\kappa \sin \theta =0$.
\end{corollary}

\begin{corollary}
$\theta -$parameter curves of the regular canal surface cannot also be lines
of curvature.
\end{corollary}

\begin{proof}
$\theta -$parameter curves are not a solution of the equation $g\tau
+h\kappa \sin \theta =0$ because when $s=s_{0}$ is a constant, $g$, $\tau $, 
$h$ and $\kappa $ are constants but $\sin \theta $ is not. For this reason, $%
\theta -$parameter curves of the canal surface cannot also be lines of
curvature.\newline
From this time, we will inspect the equation $g\tau +h\kappa \sin \theta =0$
which solves the problem when $s-$parameter curves of the canal surface are
also lines of curvature. Here, we have a look at two different cases for
this equation. These cases are as below.\newline
\textbf{The Case 1} If $s-$parameter curves which are also lines of
curvature $\theta _{0}=0$ and $\theta _{0}=\pi $ are replaced in the
equation $g\tau +h\kappa \sin \theta =0$, it follows $g\tau =0$. Since $%
g\neq 0$, $\tau =0$ i.e. the spine curve $C(s)$ becomes planar.\newline
\textbf{The Case 2} We go over $s-$parameter curves which are also lines of
curvature except for $\theta _{0}=0,\pi $. If we substitute $g=r\sqrt{%
1-r^{^{\prime }}{}^{2}}$ and $h=rr^{^{\prime }}$ in the equation $g\tau
+h\kappa \sin \theta =0$, we obtain%
\begin{equation*}
g\tau +h\kappa \sin \theta =0
\end{equation*}%
\begin{equation*}
\tau r\sqrt{1-r^{^{\prime }}{}^{2}}=-\kappa rr^{^{\prime }}\sin \theta
\end{equation*}%
\begin{equation}
\tau \sqrt{1-r^{^{\prime }}{}^{2}}=-\kappa r^{^{\prime }}\sin \theta . 
\tag{3.3}
\end{equation}%
If we take square of both sides in Eq (3.3) and then arrange, we get the
radius function $r(s)$ of the canal surface as follows.%
\begin{equation*}
\left( \tau ^{2}+\kappa ^{2}\sin ^{2}\theta \right) r^{^{\prime }2}=\tau ^{2}
\end{equation*}%
\begin{equation}
r^{^{\prime }}=\frac{\left \vert \tau \right \vert }{\sqrt{\tau ^{2}+\kappa
^{2}\sin ^{2}\theta }}  \notag
\end{equation}%
\begin{equation}
r(s)=\dint \frac{\left \vert \tau (s)\right \vert }{\sqrt{\tau
^{2}(s)+\kappa ^{2}(s)\sin ^{2}\theta }}ds+c.  \tag{3.4}
\end{equation}
\end{proof}

\begin{corollary}
(1) Let the spine curve $C(s)$ be a general helix. Then $s-$parameter curves
of the regular canal surface are also lines of curvature if and only if the
canal surface is generated by a moving sphere with the linear radius
function $r(s)=as+c$, $c>0$.\newline
(2) Let the spine curve $C(s)$ be a circular helix. For $s-$parameter curves
which are also lines of curvature, the canal surface is generated by a
moving sphere with the linear radius function $r(s)=\dfrac{b}{\sqrt{%
b^{2}+a^{2}\sin ^{2}\theta }}s+c$; $c,\tau >0$.
\end{corollary}

\begin{proof}
(1) By Eq (3.4), we have $r(s)=\dint \dfrac{\left \vert \tau (s)\right \vert 
}{\sqrt{\tau ^{2}(s)+\kappa ^{2}(s)\sin ^{2}\theta }}ds+c$. As the spine
curve $C(s)$ is a general helix, the ratio of its curvatures $\dfrac{\tau }{%
\kappa }=\tan \phi $ is a constant. Because $\theta $ is also a constant, we
obtain the radius function $r(s)$ as shown below.%
\begin{eqnarray*}
r(s) &=&\dint \sqrt{\dfrac{1}{1+\dfrac{\kappa ^{2}(s)}{\tau ^{2}(s)}\sin
^{2}\theta }}ds \\
r(s) &=&\frac{s}{\sqrt{1+\cot ^{2}\phi \sin ^{2}\theta }}+c,\text{ }c>0.
\end{eqnarray*}%
Then the radius function have the linear equation like $r(s)=as+c$ where%
\begin{equation*}
a=\dfrac{1}{\sqrt{1+\cot ^{2}\phi \sin ^{2}\theta }}.
\end{equation*}%
(2)\textbf{\ }If the spine curve $C(s)$ is a circular helix, it can be
parametrized as%
\begin{equation*}
C(s)=\left( a\cos \dfrac{s}{d},a\sin \dfrac{s}{d},b\dfrac{s}{d}\right) ,
\end{equation*}%
where $a=\dfrac{\kappa }{\kappa ^{2}+\tau ^{2}}$, $b=\dfrac{\tau }{\kappa
^{2}+\tau ^{2}}$, $d^{2}=a^{2}+b^{2}$. Because of the fact that the
curvatures $\kappa =\dfrac{a}{d^{2}}$ and $\tau =\dfrac{b}{d^{2}}$, from Eq
(3.4) it gathers%
\begin{equation*}
r(s)=\dfrac{b}{\sqrt{b^{2}+a^{2}\sin ^{2}\theta }}s+c;\text{ }c,\tau >0.
\end{equation*}%
E. Salkowski $[3]$ studied the family of space curves with constant
curvature and non-constant torsion and then Monterde $[5]$ characterized
them as space curves with constant curvature and whose normal vector makes a
constant angle with a fixed line. According to this, the curvature and the
torsion of a Salkowski curve can be given as $\kappa (s)\equiv 1$ and $\tau
(s)=\tan (\arcsin (ms))$ where $m=\dfrac{1}{\tan \phi }$ and $\phi $ is the
angle between principal normal of the curve and the fixed line.
\end{proof}

\begin{corollary}
Let the spine curve $C(s)$ be a Salkowski curve. Then $s-$parameter curves ($%
\theta =$constant) on the regular canal surface are also lines of curvature
if and only if the canal surface is generated by a moving sphere with the
radius function%
\begin{equation*}
r(s)=\frac{1}{\cos ^{2}\theta }\sqrt{\cos ^{2}\theta s^{2}+\sin ^{2}\theta
\tan ^{2}\phi }+c.
\end{equation*}
\end{corollary}

\begin{proof}
Suppose that the spine curve $C(s)$ is a Salkowski curve. In that case, $%
\kappa (s)\equiv 1$ and $\tau (s)=\tan (\arcsin (ms))$. By using Eq (3.4) we
obtain the radius function as%
\begin{equation*}
r(s)=\int \frac{\tan (\arcsin (ms))}{\sqrt{\tan ^{2}(\arcsin (ms))+\sin
^{2}\theta }}ds.
\end{equation*}%
Here, if we make the changing of variable $x=\arcsin (ms)$ we reach%
\begin{eqnarray*}
&=&\frac{1}{m}\int \frac{\sin x}{\sqrt{\tan ^{2}x+\sin ^{2}\theta }}dx \\
&=&\frac{1}{m}\int \frac{\sin x\cos x}{\sqrt{\sin ^{2}x+\sin ^{2}\theta \cos
^{2}x}}dx.
\end{eqnarray*}%
Again, for the changing of variable $\sin x=t$, we get%
\begin{equation*}
=\frac{1}{m}\int \frac{tdt}{\sqrt{\cos ^{2}\theta t^{2}+\sin ^{2}\theta }}.
\end{equation*}%
In the end,%
\begin{equation*}
\int \frac{\tan (\arcsin (ms))}{\sqrt{\tan ^{2}(\arcsin (ms))+\sin
^{2}\theta }}ds=\frac{1}{m\cos ^{2}\theta }\sqrt{\cos ^{2}\theta
m^{2}s^{2}+\sin ^{2}\theta }+c.
\end{equation*}%
Then, the radius function of the canal surface which is generated by the
Salkowski curve is%
\begin{equation*}
r(s)=\frac{1}{\cos ^{2}\theta }\sqrt{\cos ^{2}\theta s^{2}+\sin ^{2}\theta
\tan ^{2}\phi }+c.
\end{equation*}
\end{proof}

\begin{corollary}
Let $s-$parameter curves on the regular canal surface be also lines of
curvature. If $r(s)$ is an increasing function, for the spine curve $C(s)$,%
\begin{equation*}
-\kappa (s)<\tau (s)<\kappa (s).
\end{equation*}
\end{corollary}

\begin{proof}
Assume that $s-$parameter curves are also lines of curvature. Then, from
Corollary (1) we have $g\tau +h\kappa \sin \theta =0$. For the regular canal
surface, $r^{^{\prime }}<1$. If we leave alone sin$\theta $ in Eq (3.3), we
get%
\begin{equation*}
\sin \theta =-\dfrac{\tau }{\kappa }\frac{\sqrt{1-r^{^{\prime }}{}^{2}}}{%
r^{^{\prime }}}.
\end{equation*}%
Additionally, because $r(s)$ is an increasing function, $0<r^{^{\prime }}<1.$
So, it concludes $\dfrac{\sqrt{1-r^{^{\prime }}(s)^{2}}}{r^{^{\prime }}(s)}%
>1 $. If we take absolute value of the above equation and use the last
inequality, since $\left \vert \sin \theta \right \vert \leq 1$ we reach $%
\left \vert \dfrac{\tau }{\kappa }(s)\right \vert <1$. Therefore, for the
spine curve $C(s)$ we obtain $-\kappa (s)<\tau (s)<\kappa (s)$.
\end{proof}

\begin{example}
It can be given an example for $s-$parameter curves that are also lines of
curvature on canal surface as follows.\newline
Vessiot $[4]$ displayed that one family of lines of curvature on canal
surface is 
\begin{equation}
\frac{d\theta }{ds}=-\tau (s)-\kappa (s)\cot \alpha (s)\sin \theta ,\text{ }%
\cos \alpha (s)=-r^{^{\prime }}(s).  \tag{3.5}
\end{equation}%
In the above equation, since $\cos \alpha (s)=-r^{^{\prime }}(s)$, it
follows $\cot \alpha (s)=\dfrac{r^{^{\prime }}(s)}{\sqrt{1-r^{^{\prime }2}(s)%
}}$. Furthermore, if $g=r(s)\sqrt{1-r^{^{\prime }2}(s)}$ and $%
h(s)=r(s)r^{^{\prime }}(s)$ are substituted in the equation $g\tau +h\kappa
\sin \theta =0$, we obtain%
\begin{equation}
\tau (s)\sqrt{1-r^{^{\prime }2}(s)}+r^{^{\prime }}(s)\kappa (s)\sin \theta
=0.  \tag{3.6}
\end{equation}%
If we arrange Eq (3.5), by Eq (3.6) we get%
\begin{equation*}
\frac{d\theta }{ds}=-\frac{\tau (s)\sqrt{1-r^{^{\prime }2}(s)}+r^{^{\prime
}}(s)\kappa (s)\sin \theta }{\sqrt{1-r^{^{\prime }2}(s)}}=0
\end{equation*}%
\begin{equation*}
\theta =\text{constant.}
\end{equation*}%
At last, we view that one family of lines of curvature given in Eq (3.5)
coincides with one family of $s-$parameter curves ($\theta =$constant) on
canal surface for our main equation $g\tau +h\kappa \sin \theta =0$.
\end{example}

\section{Conclusions}

In this paper, we observed when parameter curves are also lines of curvature
for a generalized tube and then we researched this property for canal
surfaces. Surprisingly, we viewed that $\theta -$parameter curves cannot be
lines of curvature on canal surfaces simultaneously. Afterwards, by taking
special spine curves we obtained the radius function of a moving sphere
which generates the canal surface and showed that one family of lines of
curvature concurs with one family of $s-$parameter curves on canal surfaces.

\end{document}